\documentclass{amsart}
\usepackage[T1]{fontenc}
\usepackage[utf8]{inputenc}

\usepackage{amsmath}
\usepackage{amssymb}
\usepackage{amsthm}
\usepackage{tkz-graph}
\usepackage{hyperref}

\theoremstyle{definition}
\newtheorem{definition}{Definition}[section]

\theoremstyle{plain}
\newtheorem{theorem}[definition]{Theorem}

\newtheorem{proposition}[definition]{Proposition}

\DeclareMathOperator{\rk}{rk}
\DeclareMathOperator{\Spec}{Spec}

\DeclareMathOperator{\MWL}{MWL}
\DeclareMathOperator{\MW}{MW}
\DeclareMathOperator{\triv}{Triv}
\DeclareMathOperator{\Hom}{Hom}

\title[Dynamics on supersingular K3 surfaces]{ Dynamics on supersingular K3 surfaces and automorphisms of Salem degree 22}
\author{Simon Brandhorst}
\address{Insitut für Algebraische Geometrie, Leibniz Universität Hannover,
	Welfengarten 1, 30167 Hannover, Germany}
\email{brandhorst@math.uni-hannover.de}
\date{July 8, 2015}
\begin{document}

\begin{abstract}
In this note we exhibit explicit automorphisms of maximal Salem degree $22$ on the supersingular K3 surface of Artin 
invariant one for all primes $p\equiv 3 \mod 4$ in a systematic way. Automorphisms of Salem degree $22$ do not lift
to any characteristic zero model.
\end{abstract}
\maketitle
\section{Introduction}
To a continuous, surjective self-map  $f$ of a compact metric space one can associate its topological entropy. Roughly speaking, this number measures how fast general points spread out under the iterations of the automorphism. By work of Gromov and Yomdin, on a compact Kähler manifold $X$ the topological entropy can be calculated in terms of the action of $f$ on the cohomology groups $H^*(X,\mathbb{Z})$. In that case, the topological entropy
is either zero or the logarithm of an algebraic number. More precisely for K3 surfaces it is a Salem number. The degree of its minimal polynomial over $\mathbb{Q}$ is called its Salem degree.  Esnault and Srinivas \cite{esnault:algebraic entropy} have extended the notion of entropy to projective
surfaces over an algebraically closed field of arbitrary characteristic. \\

Salem numbers of degree 22 were used by McMullen \cite{McMullen:siegel} to construct K3 surfaces
admitting an automorphism with Siegel disks. These are domains on which $f$ acts by an irrational rotation.
Since the Salem degree of a projective surface is bounded by its Picard number $\rho \leq 20$, McMullen's examples can not be projective. They remain abstract objects.\\

However, in positive characteristic there exist K3 surfaces with Picard number $22$ so there may exist automorphisms of Salem degree $22$. As pointed out by Esnault and Oguiso \cite{esnault oguiso:non liftability}, a specific feature of such automorphisms is that they do not lift to any characteristic zero model. 
One such surface is the supersingular K3 surface $X(p)$ of Artin invariant one defined over $\overline{\mathbb{F}}_p$. 
Abstractly, Blanc, Cantat \cite{blanc cantat:non lift}, Esnault, Oguiso and Yu \cite{esnault oguiso yu:maximum salemdeg} have proved the existence of automorphisms of Salem degree 22 on $X(p)$ for $p\neq 5,7,13$ 
while Shimada \cite{shimada:dynamics} has exhibited such automorphisms 
on every supersingular K3 surface in all odd characteristics $p \leq 7919$ using double plane involutions.
Meanwhile Schütt \cite{schuett:supersingular} has exhibited explicit automorphisms of Salem degree 22 
on $X(p)$ for all $p\equiv 2 \mod 3$ using elliptic fibrations.\\

Building on his methods we obtain the main result of this paper:

\begin{theorem}\label{thm:salem deg22}.
 The supersingular K3 surface $X(p)$ of Artin invariant one admits explicit automorphisms of Salem degree 22 for all 
 primes $p \equiv 3 \mod 4$. Such automorphisms do not lift to any characteristic zero model of $X(p)$.
\end{theorem}

Let $X$ be the K3 surface over $\mathbb{Q}$ defined by $y^2=x^3+t^3(t-1)^2x$. For $p \equiv 3 \mod 4$ its specialization mod $p$ is the supersingular K3 surface $X(p)$ of Artin invariant one.
The automorphisms are constructed in the following steps:
\begin{itemize}
 \item[1] Find generators of $NS(X)$ using the elliptic fibration.
 \item[2] Complement the generators of $NS(X)$ by two sections $P,R$ to generators of $NS(X(p))$ using  a purely inseparable base change.
 \item[3] Compute the intersection matrix of $NS(X(p))$.
 \item[4] Search for extended ADE-configurations of $(-2)$-curves in $NS(X)$. These induce elliptic fibrations on $X$.
 \item[5] $P$ and $R$ induce sections of the new elliptic fibration.
 \item[6] The sections induce automorphisms of $X(p)$. Compute their action on $NS(X(p))$.
 \item[7] Compose automorphisms obtained from different fibrations to obtain one of the desired Salem degree. 
\end{itemize}

\section{Supersingular K3 surfaces}
A K3 surface over an algebraically closed field $k=\overline{k}$, is a smooth surface with
\[h^1(X,\mathcal{O}_X)=0,\quad \omega_X\cong \mathcal{O}_X.\]
Common examples are smooth quartics in $\mathbb{P}^3$ and double covers of $\mathbb{P}^2$ branched over a smooth
sextic. The group of divisors modulo algebraic equivalence is called Néron Severi group $NS(X)$. Its rank is called the Picard number. 
 Equippied with the intersection pairing $NS(X)$ is a lattice (see Section \ref{sect:lattices}).
A singular K3 surface over $\mathbb{C}$ is one whose Picard number 
\[\rho(X)=20=h^{1,1}(X),\] which is the 
maximal possible. Here 'singular' is meant in the sense of exceptional rather than non-smooth. 
In positive characteristic, however, the Picard number is only bounded by the second Betti number $b_2(X)=22$ and 
K3 surfaces reaching the maximum possible Picard number 
\[\rho(X)=22=b_2(X)\]
are called supersingular.
Supersingular K3 surfaces are classified according to their 
Artin invariant $\sigma$ defined by $\det NS(X)=-p^{2\sigma}, \sigma \in \{1,...10\}$ \cite{artin}.
By work of Ogus \cite{Ogus:supersingular}, there is a unique K3 surface of Artin invariant $\sigma=1$, over 
$k=overline{k}$ of characteristic $p$. We shall denote it by $X(p)$.\\

Supersingular K3 surfaces arise from singular K3 surfaces as follows:
\begin{proposition}\cite{schuett:supersingular}\label{prop:supersingular}
 Let $X$ be a singular K3 surface defined over a number field $L$ and $d=\det NS(X)$. If $\mathfrak{p}$ is a prime of good reduction above $p\in \mathbb{N}$,
 then $X_p:=X\times \Spec \overline{\mathbb{F}}_p$ is supersingular if $p$ is inert in $\mathbb{Q}(\sqrt{d})$.
\end{proposition}
As noticed by Shimada \cite{shimada:supersing}, if $\det NS(X)$ is coprime to $p$, then the Artin invariant is $\sigma=1$. The reason for this is that $NS(X) \hookrightarrow NS(X_p)$ implies $\sigma = 1$. 

\section{Automorphisms and Dynamcis}

A \emph{Salem polynomial} is an irreducible monic polynomial $S(x)\in \mathbb{Z}[x]$ of degree $2d$ such that its complex roots are of the form
\[ a>1,\; a^{-1}<1, \quad \alpha_i,\overline{\alpha}_i \in S^1, \; i\in \{1, ... ,d-1\}.\]
The unique positive root $a$ outside the unit disk is called
\emph{Salem number}. \\

The following theorem is due to McMullen \cite{McMullen:siegel} over $\mathbb{C}$ and Esnault and Srinivas \cite{esnault:algebraic entropy} in positive characteristic.
\begin{theorem}\label{thm:salem factor}
  Let $X$ be a smooth projective K3 surface over an algebraically closed field $k=\overline{k}$ and $f\in Aut(X)$ an automorphism.
  Then the characteristic polynomial of $f^*|H^2_{\acute{e}t}(X,\mathbb{Q}_l(1)), \; (char \; k \neq l)$ factors into cyclotomic polynomials and at most one
  Salem polynomial. The Salem factor occurs in $f^*|NS(X)$.
\end{theorem}
Notice that over $\mathbb{C}$ we can work equally well with the singular cohomology groups $H^2(X,\mathbb{Z})$
and prove the result via Hodge decomposition. The theorem motivates the following definition.
Given a smooth projective surface over an algebraically closed field and an automorphism $f:X \rightarrow X$, 
the entropy $h(f)$ of $f$ is defined as 
\[h(f):=\log r(f^*|NS(X))\]
where $r(f^*|NS(X))$ is the spectral radius of  the linear map $f^*|NS(X)\otimes \mathbb{Q}$, that is, the maximum of the absolute values 
of its complex eigenvalues. The entropy is either zero or the logarithm of a Salem number $a$. The degree of
the Salem polynomial corresponding to $a$ is called the Salem degree of $f$.\\

By work of H. Esnault and V. Srinivas \cite{esnault:algebraic entropy} this is compatible with
the definition of topological entropy for smooth complex projective surfaces.\\

Over $\mathbb{C}$, due to Hodge theory, the rank of the Néron-Severi group is bounded by $20=h^{1,1}(X)$. 
Hence an automorphism on a complex projective K3 surface has Salem degree at most $20$. 
In the non-projective case McMullen \cite{McMullen:siegel} has constructed K3 surfaces with
automorphisms of Salem degree $22$. In positive characteristic, however, K3 surfaces may have 
Picard rank $22$ and there automorphisms of Salem degree $22$ occur. A specific feature of such automorphisms is that
they do not lift to any characteristic zero model.

\section{Lattices}\label{sect:lattices}
In this section we fix our notation concerning lattices.
A lattice $L$ is a finitely generated free $\mathbb{Z}$ module equipped with a non degenerate symmetric bilinear pairing
$L \times L \rightarrow \mathbb{Q}$. If it is integer valued, then the lattice is called integral. The pairing identifies 
\[\Hom(L,\mathbb{Z}) \cong L^*:=\{x \in L \otimes \mathbb{Q} | x.L \subseteq \mathbb{Z} \}.\] 
If $L$ is integral, it induces an embedding $L \hookrightarrow L^*$. The finite quotient $A_L:=L^*/L$ is called the discriminant group of $L$. Given a $\mathbb{Z}$-basis $(b_i)$ the determinant of the Gram matrix $(b_i.b_j)_{ij}$ is independent of the choice of basis and called the determinant of $L$. Its absolute value is the order of the discriminant group. In this basis the dual lattice is generated by the columns of $(b_{ij})^{-1}$. 
Given an embedding $M \hookrightarrow L$ of integral lattices of the same rank, $A_L$ is a subquotient of $A_M$ and we get
\[\det M=[L:M]^2 \det L .\]
A lattice is called unimodular if it is integral and has $|\det L|=1$. An embedding $M\hookrightarrow L$ of lattices is called primitive if $L/M:=L/im(M)$ is torsion free. 
If $M\hookrightarrow L$ is a primitive embedding of integral lattices and $L$ unimodular, then 
$A_M \cong A_{M^\perp}$, where $M^\perp$ is the orthogonal complement of $M$ in $L$. 
We define the length of a finite abelian group by its minimum number of generators. Note that
$l(A_L)\leq \rk L$. A lattice is called $p$-elementary if $pA_L=0$ or equivalently if $A_L$ is an $\mathbb{F}_p$ vector
space.

\section{Elliptic Fibrations on K3 Surfaces}
A genus one fibration on a surface X is a surjective map 
\[\pi: X \rightarrow  C\]
to a smooth curve $C$ such that the generic fiber is a smooth curve of genus one.
We will call it an elliptic fibration, if the additional data of a section $O$ of $\pi$ is given.
Indeed, all genus one fibrations occurring throughout this note admit a section.
This turns the generic fiber of an elliptic fibration $E(K)$ into an elliptic curve over $K=k(C)$, the function field of $C$.
For a K3 surface $C=\mathbb{P}^1$ is the only possibility.
There is a one to one correspondence between rational points of $E$ and sections of $\pi$. 
Both these sets are abelian groups which we will call the Mordell-Weil group of the fibration. It is denoted by
$\MW(X,\pi,O)$ where $\pi$ and $O$ are suppressed from notation if confusion is unlikely. The addition on
$\MW$ is denoted by $\oplus$.\\

A good part of $NS$ is readily available: the trivial lattice 
\[\triv(X,\pi,O):=\langle O,\mbox{fiber components}\rangle_\mathbb{Z}. \]
If $O$ and $\pi$ are understood, we will suppress them from notation.
By results of Kodaira \cite{Kodaira} and Tate \cite{tate:algorithm}, the trivial lattice decomposes as an orthogonal direct sum of a hyperbolic plane spanned by $O$ together with the fiber $F$ and negative definite 
root lattices of type $ADE$ consisting of fiber components.
Note that the singular fibers (except in some cases in characteristics 2 and 3) are determined by the
$j$-invariant and discriminant of the elliptic curve $E(K)$.\\

An advantage of elliptic fibrations is that they structure the Néron-Severi group into
sections and fibers as given by the following theorem.
\begin{theorem}\cite{shioda:mwl}\label{thm:mwl}
 There is a group isomorphism 
 \[\MW(X)\cong NS(X)/\triv(X).\]
\end{theorem}
The Mordell-Weil group can be equipped with a positive definite symmetric $\mathbb{Q}$-valued bilinear form via the 
orthogonal projection with respect to $\triv(X)$ in $NS(X) \otimes \mathbb{Q}$ and switching sign.
Explicitly, it is given as follows: 
Let $P,Q \in \MW(X,O)$ and denote by $R$ the set of singular fibers of the fibration, then 
\[\langle P , Q \rangle:= \chi (\mathcal{O}_X) + P.O + Q.O -P.Q - \sum_{\nu \in R}c_\nu(P,Q)\]
\[ \langle P , P \rangle:= 2 \chi(\mathcal{O}_X) + 2 P.O - \sum_{\nu \in R}c_\nu(P,Q)\]
where the dot denotes the intersection product on the smooth surface $X$. The term $c_\nu(P,Q)$ 
is the local contribution at a singular fiber, given as follows. If one of the two sections involved meets the same component of the singular fiber $\nu$ as the zero section, then the contribution at $\nu$ is zero. If this is not the case, the contribution is non zero and depends
of the fiber type. We only need the types $III$, $III^*$ and $I_0^*$, for the others consult \cite{shioda:mwl}. If $\nu$ is of type 
$III^*$ (resp. $III$), the contribution is equal to $3/2$ (resp. $1/2$) if $P$ and $Q$ meet the same component of $\nu$ and zero otherwise.
For $\nu$ of type $I^*_0$ we have $c_\nu(P,Q)=1$ if they meet in the same component and $1/2$ otherwise.
Equipped with this pairing $\MWL(X):=\MW(X)/$torsion is a positive definite lattice. In general it is not integral.

\section{Isotrivial Fibration}
In this section we will compute generators of the Néron-Severi group as well as their intersection matrix.
\begin{proposition}
Let $X$ be the K3 surface over $\mathbb{C}$ defined by the Weierstrass equation 
\[X: \quad y^2=x^3+t^3(t-1)^2x.\]
Then its Néron-Severi group is generated by fiber components, the zero section and the 2-torsion section $Q=(0,0)$. 
It is of rank 20 and determinant $-4$.
\end{proposition}
\begin{proof}
By the theory of Mordell-Weil lattices, $NS(X)$ is spanned by fiber components and sections.
The elliptic fibration has $j$-invariant equal to $1728$ and discriminant $\Delta=-2^6t^9(t-1)^6$.
Using the classification of singular fibers by Kodaira and Tate, we get that the fibration has 
two fibers of type $III^*$ over $t=0,\infty$ and one fiber of type $I_0^*$ over $t=1$.
Hence the trivial lattice $L$ is of the form
$L\cong U\oplus 2 E_7 \oplus D_4$. Since it is of the maximum possible rank $20$ the fibration admits no section of infinite order.
The determinant $-16$ of the trivial lattice implies that only 2- or 4-torsion may appear.
Obviously, $Q=\{x=y=0\}$ is the only 2-torsion section and 4-torsion may not occur due to the singular fibers. 
Alternatively, the reader may note that additional torsion turns $NS(X)$ into a unimodular lattice of signature $(1,19)$. Such a lattice does not exist.
\end{proof}
Note that the $j$-invariant $j=1728$ is constant. Hence all smooth fibers are isomorphic - such a fibration is called
isotrivial.

\begin{proposition}
 For $p \equiv 3 \mod 4$ the surface $X(p):=X\otimes  \overline{\mathbb{F}}_p$ is the supersingular K3 surface of Artin invariant one.
\end{proposition}
\begin{proof}
The singular K3 surface $X$ has good reduction at $p\neq2$ and $\det NS(X)=-4$. A prime $p$ is inert in $\mathbb{Q}(\sqrt{-1})$ iff $p \equiv 3 \mod 4$.
Thus, by Proposition \ref{prop:supersingular}, for all $p \equiv 3 \mod 4$ the K3 surface $X(p):=X\otimes \overline{\mathbb{F}}_p$ is supersingular, that is, $\rk NS(X(p))=22$. It is known that $NS(X(p))$ is a $p$-elementary lattice of determinant 
$-p^{2\sigma}$ where $\sigma \in \{1,...,10\}$ is called the Artin invariant of $X(p)$. 
Following an argument by Shimada \cite{shimada:supersing} we show that $\sigma=1$:
$NS(X)\hookrightarrow NS(X(p))$ Therefore $NS(X) \oplus NS(X)^\perp \hookrightarrow NS(X(p))$.
Since $p\neq 2$, the $p$-part of $A_{NS(X) \oplus NS(X)^\perp}=A_{NS(X)} \oplus A_{NS(X)^\perp}$ equals that of
$A_{NS(X)^\perp}$. Hence it is of length at most two. But the $p$-part of $A_{NS(X(p))}$ is a quotient of this.
We conclude that its $p$-part has length at most two. On the other hand $\sigma \in \{1...10\}$ which implies that 
$2\sigma=l(A_{NS(X(p))})\geq 2$.
\end{proof}

Our next task is to work out generators of $NS(X(p))$ for $p \equiv 3 \mod 4$. By Theorem \ref{thm:mwl} it is generated by sections and fiber components. Reducing $j$ and $\Delta$ $\mod p$, we see that
the fiber types are preserved mod $p$ (even in case $p=3$ cf. \cite{silverman advanced topics}). Hence sections of infinite order must appear. Generally it is hard to compute sections of an elliptic fibration. For this special fibration there is a trick involving a purely inseparable base change of degree $p$ turning $X(p)$ into a Zariski surface.\\

\begin{proposition}
 Let $p=4n+3$ be a prime number. Then the Néron-Severi group of $X(p)$ defined by $y^2=x^3+t^3(t-1)^2x$ over $\overline{\mathbb{F}}_p$ is generated by the sections $O,Q,P,R$ and fiber components, where $O$ denotes the section at infinity, $Q=(0,0)$ is the 2-torsion section, $\zeta^4=-1$ and
 \begin{align*}
P: \;\; &x=\zeta^2t^{2n+3}(t-1),		&& y=\zeta^3t^{n+3}(t-1)^{2n+3},\\
R: \;\; &x=-\zeta^2t^{2n+3}(t-1), 	&&y=-\zeta t^{n+3}(t-1)^{2n+3} .
\end{align*}
We have the following intersection numbers, symmetric in $P$ and $R$:
\[O.P=Q.P=n, P.R=2n. \]
\end{proposition}
\begin{proof}
One can check directly that $P$ and $R$ are sections of the elliptic fibration and the patient reader may calculate their intersection numbers by hand.  
Since $X(p)$ has Artin invariant $\sigma=1$, $\det NS(X)=-p^{2\sigma}=-p^{2}$.  All that remains is to compute the intersection matrix of these four sections and the fiber components. One can check that it has a $22 \times 22$ minor of determinant $-p^2$. This minor corresponds to a basis of $NS(X(p))$.
\end{proof}

For later reference we fix the the following $\mathbb{Z}$-basis of the Néron-Severi group, where the fiber components
are sorted as indicated in Figure \ref{ell1}, and $e_{20}$ is distinguished by $e_{20}.P=1$.
\[(O,F,Q,E_7 (t=\infty),E_7 (t=0),A_3 (\subseteq D_4,t=1),e_{21}=P,e_{22}=R)\]
Blue vertices and edges belong to fibers. Replacing $Q$ by one of the missing components of the $I_0^*$ fiber
results in a $\mathbb{Q}$ instead of a $\mathbb{Z}$-basis. This is predicted by Theorem \ref{thm:mwl}.\\

\begin{figure}
\begin{center}
\begin{tikzpicture}[scale=0.5]
\GraphInit[vstyle=Normal]
\SetGraphUnit{1}
\SetVertexMath
\tikzset{VertexStyle/.style= {fill=blue, inner sep=1pt, shape=circle}}
\Vertex[LabelOut,Lpos=180,x=0,y=0]{11}
\Vertex[LabelOut,Lpos=180,x=0,y=1]{12}
\Vertex[LabelOut,Lpos=180,x=0,y=2]{13}
\Vertex[LabelOut,x=0,y=3]{14}
\Vertex[LabelOut,Lpos=180,x=-1,y=3]{15}
\Vertex[LabelOut,Lpos=180,x=0,y=4]{16}
\Vertex[LabelOut,Lpos=180,x=0,y=5]{17}
\Vertex[NoLabel,Lpos=180,x=0,y=6]{e18}

\SetUpEdge[color=blue]
\Edges(11,12,13,14,16,17,e18)
\Edge(14)(15)

\tikzset{VertexStyle/.style= {fill=black, inner sep=1pt, shape=circle}}
\Vertex[Math,LabelOut,Lpos=-90,x=3,y=0]{Q}
\Vertex[Math,LabelOut,Lpos=90,x=3,y=6]{O}

\tikzset{VertexStyle/.style= {fill=blue, inner sep=1pt, shape=circle}}
\Vertex[LabelOut,x=6,y=0]{4}
\Vertex[LabelOut,x=6,y=1]{5}
\Vertex[LabelOut,x=6,y=2]{6}
\Vertex[LabelOut,Lpos=180,x=6,y=3]{7}
\Vertex[LabelOut,x=7,y=3]{8}
\Vertex[LabelOut,x=6,y=4]{9}
\Vertex[LabelOut,x=6,y=5]{10}
\Vertex[NoLabel,x=6,y=6]{e15}

\Edges(4,5,6,7,9,10,e15)
\Edge(7)(8)
\SetUpEdge[color=black]
\Edges(11,Q,4)
\Edges(e15,O,e18)

\Vertex[LabelOut,x=3,y=2]{19}
\Vertex[LabelOut,Lpos=40,Ldist=-0.15cm,x=3,y=3]{18}
\Vertex[NoLabel,x=3,y=4]{d3}
\Vertex[NoLabel,x=2,y=3]{d4}
\Vertex[LabelOut,x=4,y=3]{20}
\Edge(Q)(19)
\Edge(d3)(O)
\SetUpEdge[color=blue]
\Edges(19,18,d3)
\Edges(d4,18,20)
\end{tikzpicture}
\end{center}
\caption{$24$ $(-2)$-curves of $X$ supporting singular fibers of type $I_0^*,2\times III^*$ (blue) and torsion sections of $\pi$.}
\label{ell1}
\end{figure}

For the intersection matrix in this basis one obtains:
\[
\left(\begin{smallmatrix}	
	-2& 1& 0&0&0&0&0&0&0&0&0&0&0&0&0&0&0&0&0&0&n&n\\
        1& 0& 1& 0&0&0&0&0&0&0&0&0&0&0&0&0&0&0&0&0&1&1\\
        0& 1& -2&1&0&0&0&0&0&0&1&0&0&0&0&0&0&0&1&0&n&n\\
        0& 0& 1&-2&1&0&0&0&0&0&0&0&0&0&0&0&0&0&0&0&1&1\\
        0& 0& 0&1&-2&1&0&0&0&0&0&0&0&0&0&0&0&0&0&0&0&0\\
        0& 0& 0&0&1&-2&1&0&0&0&0&0&0&0&0&0&0&0&0&0&0&0\\
        0& 0& 0&0&0&1&-2&1&1&0&0&0&0&0&0&0&0&0&0&0&0&0\\
        0& 0& 0&0&0&0&1&-2&0&0&0&0&0&0&0&0&0&0&0&0&0&0\\
        0& 0& 0&0&0&0&1&0&-2&1&0&0&0&0&0&0&0&0&0&0&0&0\\
        0& 0& 0&0&0&0&0&0&1&-2&0&0&0&0&0&0&0&0&0&0&0&0\\
        0& 0& 1&0&0&0&0&0&0&0&-2&1&0&0&0&0&0&0&0&0&0&0\\
        0& 0& 0&0&0&0&0&0&0&0&1&-2&1&0&0&0&0&0&0&0&0&0\\
        0& 0& 0&0&0&0&0&0&0&0&0&1&-2&1&0&0&0&0&0&0&0&0\\
        0& 0& 0&0&0&0&0&0&0&0&0&0&1&-2&1&1&0&0&0&0&0&0\\
        0& 0& 0&0&0&0&0&0&0&0&0&0&0&1&-2&0&0&0&0&0&0&0\\
        0& 0& 0&0&0&0&0&0&0&0&0&0&0&1&0&-2&1&0&0&0&0&0\\
        0& 0& 0&0&0&0&0&0&0&0&0&0&0&0&0&1&-2&0&0&0&0&0\\
        0& 0& 0&0&0&0&0&0&0&0&0&0&0&0&0&0&0&-2&1&1&0&0\\
        0& 0& 1&0&0&0&0&0&0&0&0&0&0&0&0&0&0&1&-2&0&0&0\\
        0& 0& 0&0&0&0&0&0&0&0&0&0&0&0&0&0&0&1&0&-2&1&0\\
        n&1&n&1&0&0&0&0&0&0&0&0&0&0&0&0&0&0&0&1&-2&2 n\\
        n&1&n&1&0&0&0&0&0&0&0&0&0&0&0&0&0&0&0&0&2 n&-2\\
\end{smallmatrix}
\right)
\]

In the remaining part of this section we will explain how the sections $P, R$ were found and give an alternative
way of computing their intersection numbers.

Recall that we assume that $p \equiv 3 \mod 4$ and write $p=4n+3$.
Consider the purely inseparable base change $t\mapsto t^p$. This changes the equation as follows.
\[y^2=x^3+t^{12n+9}(t-1)^{8n+6}x\]
We minimize this equation using the birational map 
\[(x,y,t)\mapsto \left(\frac{x}{t^{6n+4}(t-1)^{4n+2}},\frac{y}{t^{9n-6}(t-1)^{6n+3}},t\right).\]
This leads to the surface $Y$ given by
\[Y: \quad y^2=x^3+t(t-1)^2x.\]
After another base change $t\mapsto t^p$ we get
\[y^2=x^3+t^{4n+3}(t-1)^{8n+6}x\]
 and minimizing the fibration 
\[\left(\frac{y}{t^{3n}(t-1)^{6n+3}}\right)^2=\left(\frac{x}{t^{2n}(t-1)^{4n+2}}\right)^3+t^3(t-1)^2\frac{x}{t^{2n}(t-1)^{4n+2}}\]
we recover $X$.\\

Instead of directly searching for sections on $X$ we will exhibit two sections on $Y$ and pull them back to $X$.
The $j$-invariant of $Y$ is still equal to $1728$, but $\Delta=-2^6t^3(t-1)^6$. 
For $p\neq 3$ this leads to two singular fibers of type $III$ over $t=0,\infty$ and a singular fiber of type $I_0^*$ over $t=1$.
The Euler number of this surface is $2e(III)+e(I_0^*)=2\cdot 3 + 6=12$. By general theory it is rational. 
The trivial lattice is isomorphic to
\[\triv (Y) \cong U \oplus 2A_1 \oplus D_4.\]
This lattice is of determinant $-16$ and rank $8$.
From the theory of elliptically fibered rational surfaces we know that $\rk NS(Y)=10$, $\det NS(Y)=-1$,
which implies that $Y$ has Mordell-Weil rank $2$.
Define \[\triv'(Y):=\triv(Y)\otimes_{\mathbb{Z}} \mathbb{Q} \cap NS(Y).\]
Then by Theorem \ref{thm:mwl} $\triv'(Y)/\triv(Y)\cong \MW_{tors}$. Since 
$\{x=y=0\}$ is a 2-torsion section, we know that 
\[\det \triv'(Y) = \det \triv(Y) / [\triv'(Y):\triv(Y)]^2\in \{-1,-4\}.\] 
As even unimodular lattices of signature $(1,7)$ do not exist, $-1$ is impossible. We conclude that $x=y=0$ is the only torsion section.\\

To find a section of infinite order, we first determine the Mordell-Weil lattice and then translate this information
to equations of the section. Since $j=1728$ the elliptic curve admits complex multiplication given by
$(x,y)\mapsto(-x,iy)$. Obviously, $Q$ and $O$ are the unique sections fixed by this action.
Hence the Mordell-Weil lattice admits an isometry of order four which, viewed as an element of $O(2)$, may only be a rotation by $\pm \tfrac{\pi}{2}$.
Up to isometry all positive definite lattices of rank $2$ admitting an isometry of order four have a Gram matrix of the form 
\[ \left( \begin{matrix}
 a& 0\\
 0 & a\\
 \end{matrix}\right)
\]
for some $a>0$.  The formula
\[\det NS(Y) = (-1)^r \det \MWL(Y) \det \triv'(Y)\]
where $r=rk \MWL(Y)=2$, $\det NS(Y)=-1$, $\det \triv'(Y)=4$, resulting from the definition of $\MWL(Y)$ via the orthogonal projection w.r.t. $\triv'(Y)$, yields $a=1/2$.
 \[\MWL (Y)\cong \left(
 \begin{matrix}
 1/2& 0\\
 0 & 1/2\\
 \end{matrix}\right).\]
We search for a section $P$ with 
\[1/2=\langle P,P\rangle
:=2 \chi + 2 P.O - c_0(P,P)-c_\infty(P,P) - c_1(P,P)\]
where $\chi=\chi(\mathcal{O}_Y)=1$ is the Euler characteristic of $Y$ and
$c_0,c_\infty \in \{0,1/2\}, c_1 \in \{0,1\}$ are the contributions at the singular fibers, $P.O \in \mathbb{N}$ is the intersection number.
This immediately implies $P.O=0$ and $c_1=1$. 
Let us assume $c_0=1/2, c_\infty=0$. 
The section $P$ can be given by $(x(t),y(t))$ where $x,y$ are rational functions.
Over the chart containing $\infty$ it is given by 
$(\hat{x}(s),\hat{y}(s))=(s^2x(1/s),s^3y(1/s))$ where $s=1/t$.
As poles of these functions correspond to intersections with the zero section,
we know that $x,y,\hat{x},\hat{y}$ are actually polynomials. Therefore, $\deg x \leq 2$ and $\deg y \leq 3$.
Furthermore, if $P$ intersects the same fiber component as the zero section, then the contribution $c_\nu$
of that fiber is zero. The other components arise by blowing up the singularities of the Weierstraß model in $\{x=y=t(t-1)=0\}$. Hence $x(0)=y(0)=x(1)=y(1)=0$. Putting this information together we obtain 
$x=at(t-1)$ and $y=t(t-1)b$ for some constant $a$ and a polynomial $b(x)$ of degree one.
A quick calculation yields the sections
\begin{align}
P': &\quad x(t)=\zeta^2t(t-1)  && y(t)=\zeta^3t(t-1)^2,\\
R': &\quad x(t)=-\zeta^2t(t-1) && y(t)=-\zeta t(t-1)^2
\end{align}
for $\zeta^4+1=0$. Note that $\MWL(Y)$ contains exactly four sections of height $1/2$.
Furthermore, $\langle P', R' \rangle=0$ (otherwise $P=\ominus R$ which is clearly false). 
The two further Galois conjugates of $\zeta$ correspond to the missing two sections $\ominus P$ and $\ominus R$.\\

Now we shall pull back $P'$ and $R'$ to $X$ via the map
\[\phi:\; X \rightarrow Y, \quad (x,y,t) \mapsto (xt^{2n}(t-1)^{4n+2},yt^{3n}(t-1)^{6n+3},t^p)\]
and get the sections
\begin{align}
P:& \quad x=\zeta^2t^{2n+3}(t-1)  && y=\zeta^3t^{n+3}(t-1)^{2n+3},\\
R:& \quad x=-\zeta^2t^{2n+3}(t-1) && y=-\zeta t^{n+3}(t-1)^{2n+3}.
\end{align}

It remains to compute the intersection numbers involving $P$ and $R$. This can be done by blowing up the singularities and then computing the intersections.
However, by applying some more of Shioda's theory we can avoid the blowing ups.
From the behavior of the height pairing under base extension (cf. \cite[Prop. 8.12]{shioda:mwl}) we know that
\[\langle P,P\rangle=\deg \phi \langle P' , P' \rangle=p \langle P' , P' \rangle=2n+3/2\]
and get the equation
\[2n+3/2=4 + 2 P.O - c_0(P,P)-c_\infty(P,P) - c_1(P,P).\]
Note that a fiber of type $III^*$ has only two simple components. 
Since P meets the same component over infinity as the zero section and passes through the singularities at $t=0,1$, we know that $c_\infty=0$, $c_0=3/2$ and $c_1=1$.
We conclude that $P.O=n$. The other intersection numbers can be calculated accordingly. 
In this way we could have calculated the intersection matrix of $NS(X(p))$ without knowing equations for 
the extra sections.\\

The section $P$ induces an automorphism of the surface by fiberwise addition. We shall denote it by $(\oplus P)$. 
The matrix of its representation on $NS$ is obtained as follows:
\begin{itemize}
 \item Compute the basis representation of the sections $Q\oplus P, 2P$ and $R \oplus P$.
 \item Any section $S$ is mapped to $P \oplus S$ under $f_*$ and the fiber is fixed.
 \item The action of $(\oplus P)$ on the Néron-Severi group preserves each singular fiber. Since it is an isometry it can 	be determined by its action on sections.
 \item Invert the resulting matrix to get $f^*=f_*^{-1}$.
\end{itemize}
A basis representation of $P \oplus Q$ is obtained as follows:
Start with $P+Q \in NS$ and subtract $nO$ such that the resulting divisor $D$ has $D.F=1$.
Add or subtract fiber components until $D$ meets each singular fiber in exactly one simple component.
Finally add multiples of $F$ such that $D^2=-2$.
For example the basis representation of the section $P\oplus R$ is given by
\[(1, 2,-2, 0, 0, 0, 0, 0, 0, 0,-3,-4,-5,-6,-3,-4,-2,-1,-2, 0, 1, 1).\]

\section{Alternative Elliptic Fibrations}
The automorphism constructed in the last section has zero entropy. 
The reason for this is that it fixes the fibers. We shall overcome this obstruction by 
combining different fibrations and their sections.\\

Reducible singular fibers of elliptic fibrations are extended ADE-configurations of $(-2)$-curves. 
Conversely, let $X$ be a K3 surface and $F$ an extended $ADE$ configuration of $(-2)$-curves.
Then the linear system $|F|$ is an elliptic pencil with the extended $ADE$ configuration as singular fiber. See
\cite{kondo:elliptic pencil},\cite{shafarevic:elliptic pencil} for details.
We use this fact to detect additional elliptic fibrations in the graph of $(-2)$-curves.
Irreducible curves $C$ are either fiber components, sections or multisections, depending on whether $C.F=0,1$ or $>1$.
Be aware that the $(-2)$ curves occur with multiplicities in $F$. Hence it is possible that $C.F>1$ even if 
$C$ meets only a single $(-2)$-curve of the configuration. 
In general an elliptic pencil does not necessarily admit a section. But once its existence is guaranteed, we may
choose it as zero section. Then, by Theorem \ref{thm:mwl}, multisections still induce sections once modified by fiber components and the zero section as
sketched above. \\

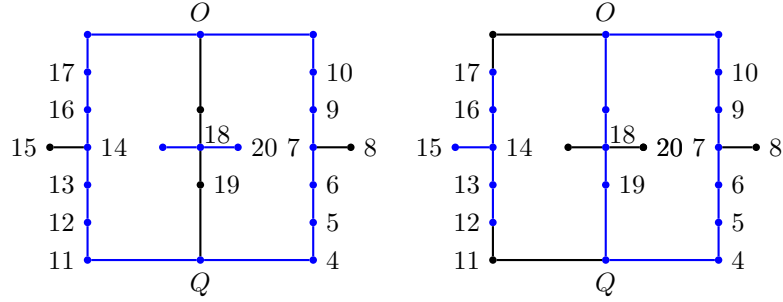
\begin{figure}
\begin{minipage}[b]{0.3333\linewidth}
\begin{tikzpicture}[scale=0.5]
\GraphInit[vstyle=Normal]
\SetGraphUnit{1}
\SetVertexMath
\tikzset{VertexStyle/.style= {fill=blue, inner sep=1pt, shape=circle}}
\Vertex[LabelOut,Lpos=180,x=0,y=0]{11}
\Vertex[LabelOut,Lpos=180,x=0,y=1]{12}
\Vertex[LabelOut,Lpos=180,x=0,y=2]{13}
\Vertex[LabelOut,x=0,y=3]{14}
\Vertex[LabelOut,Lpos=180,x=0,y=4]{16}
\Vertex[LabelOut,Lpos=180,x=0,y=5]{17}
\Vertex[NoLabel,Lpos=180,x=0,y=6]{e18}
\Vertex[LabelOut,Lpos=-90,x=3,y=0]{Q}
\Vertex[LabelOut,Lpos=90,x=3,y=6]{O}
\Vertex[LabelOut,x=6,y=0]{4}
\Vertex[LabelOut,x=6,y=1]{5}
\Vertex[LabelOut,x=6,y=2]{6}
\Vertex[LabelOut,Lpos=180,x=6,y=3]{7}
\Vertex[LabelOut,x=6,y=4]{9}
\Vertex[LabelOut,x=6,y=5]{10}
\Vertex[NoLabel,x=6,y=6]{e15}
\Vertex[LabelOut,Lpos=40,Ldist=-0.15cm,x=3,y=3]{18}
\Vertex[NoLabel,x=2,y=3]{d4}
\Vertex[LabelOut,x=4,y=3]{20}

\tikzset{VertexStyle/.style= {fill=black, inner sep=1pt, shape=circle}}
\Vertex[LabelOut,Lpos=180,x=-1,y=3]{15}
\Vertex[LabelOut,x=3,y=2]{19}
\Vertex[NoLabel,x=3,y=4]{d3}
\Vertex[LabelOut,x=7,y=3]{8}

\SetUpEdge[color=blue]
\Edges(11,12,13,14,16,17,e18)
\Edges(4,5,6,7,9,10,e15)
\Edges(11,Q,4)
\Edges(e15,O,e18)
\Edges(d4,18,20)

\SetUpEdge[color=black]
\Edge(14)(15)
\Edge(7)(8)
\Edge(Q)(19)
\Edge(d3)(O)
\Edges(19,18,d3)
\SetUpEdge[color=blue]
\end{tikzpicture}
\end{minipage}
\qquad \quad
\begin{minipage}[b]{0.3333\linewidth}
\begin{tikzpicture}[scale=0.5]
\GraphInit[vstyle=Normal]
\SetGraphUnit{1}
\SetVertexMath
\tikzset{VertexStyle/.style= {fill=blue, inner sep=1pt, shape=circle}}
\Vertex[LabelOut,Lpos=180,x=0,y=1]{12}
\Vertex[LabelOut,Lpos=180,x=0,y=2]{13}
\Vertex[LabelOut,x=0,y=3]{14}
\Vertex[LabelOut,Lpos=180,x=0,y=4]{16}
\Vertex[LabelOut,Lpos=180,x=0,y=5]{17}
\Vertex[LabelOut,Lpos=-90,x=3,y=0]{Q}
\Vertex[LabelOut,Lpos=90,x=3,y=6]{O}
\Vertex[NoLabel,x=6,y=6]{e15}
\Vertex[LabelOut,x=6,y=0]{4}
\Vertex[LabelOut,x=6,y=1]{5}
\Vertex[LabelOut,x=6,y=2]{6}
\Vertex[LabelOut,Lpos=180,x=6,y=3]{7}
\Vertex[LabelOut,x=6,y=4]{9}
\Vertex[LabelOut,x=6,y=5]{10}
\Vertex[LabelOut,Lpos=40,Ldist=-0.15cm,x=3,y=3]{18}
\Vertex[LabelOut,Lpos=180,x=-1,y=3]{15}
\Vertex[LabelOut,x=3,y=2]{19}
\Vertex[NoLabel,x=3,y=4]{d3}

\tikzset{VertexStyle/.style= {fill=black, inner sep=1pt, shape=circle}}
\Vertex[LabelOut,x=4,y=3]{20}
\Vertex[NoLabel,x=2,y=3]{d4}
\Vertex[LabelOut,x=4,y=3]{20}
\Vertex[LabelOut,x=7,y=3]{8}
\Vertex[NoLabel,Lpos=180,x=0,y=6]{e18}
\Vertex[LabelOut,Lpos=180,x=0,y=0]{11}

\SetUpEdge[color=blue]
\Edges(12,13,14,16,17)
\Edges(4,5,6,7,9,10,e15)
\Edges(Q,4)
\Edges(e15,O)
\Edge(Q)(19)
\Edge(O)(d3)
\Edges(19,18,d3)
\Edge(14)(15)

\SetUpEdge[color=black]
\Edges(d4,18,20)
\Edge(7)(8)
\Edge(Q)(11)
\Edge(O)(e18)
\Edge(11)(12)
\Edge(17)(e18)
\end{tikzpicture}
\end{minipage}
\caption{$\pi '$ with $I_{16}$ and $I_4$ fibers 
and $\pi''$ with $I_{12}$ and $IV^*$ fibers. }\label{fig:alternative ellfib}
\end{figure}

The first fibration $\pi'$ is induced by the outer circle of $(-2)$-curves which is a singular fiber of type $I_{16}$.
There is a second singular fiber of type $I_4$. Three of its components are visible in Figure \ref{fig:alternative ellfib}. The curve $e_8$ ($=$ vertex labeled by '8')  is a section since it meets $I_{16}$ exactly in a simple component. We take it as zero section. Then the torsion sections are $e_{15},e_{18}$ and $e_{19}$.
The second fibration $\pi''$ is induced by the right inner circle of $(-2)$-curves. It has singular fibers of type $I_{12}$ and $IV^*$ and again we take $e_8$ as zero section. A simple component of the $IV^*$ fiber is not visible in
Figure \ref{fig:alternative ellfib}.
In both cases $P$ is a multi section and induces a section of each fibration denoted by $P'$ and $P''$.
For example the class of $P'\in NS(X(p))$ is given by
\begin{align*}
 P'=(&n,n,n+1,2,2-n,-2n+2,2-3n,1-2n,1-2n,-n,0,\\
&-2n-2,-3-3n,-2n-2,-2n-2,-1-n,0,0,0,1,0).
\end{align*}

\section{Salem degree 22 automorphism}
 
 We consider the automorphism 
 \[f:=(\oplus R)\circ (\oplus P)\circ (\oplus P')\circ (\oplus P'').\] 
 on $X(p)$ for $p \equiv 3 \mod 4$.
 Using a computer algebra system one computes the characteristic polynomial of $f^*|NS(X(p))$:
\[\mu (f^*)=x^{11}g(x+1/x),\]
where
 \begin{align*}
 g(x):=& 8 n ^2 + 88 n + 67 \\
     + &(-88 n ^3 - 392 n^2  - 976 n - 574) x\\
     + &(-232 n^3  - 1474 n^2  - 2854 n - 1464) x^2\\
     + &(534 n^3  + 2526 n^2  + 4605 n + 2359) x^3\\
     + &(578 n^3  + 3415 n^2  + 6196 n + 3062) x^4\\
     + &(-568 n^3  - 2749 n^2  - 4587 n - 2245) x^5\\
     + &(-466 n^3  - 2689 n^2  - 4681 n - 2253) x^6\\
     + &(206 n^3  + 1014 n^2  + 1600 n + 770) x^7\\
     + &(148 n^3  + 849 n^2  + 1426 n + 670) x^8\\  
     + &(-24 n^3  - 120 n^2  - 182 n - 91) x^9\\
     + &(-16 n^3  - 92 n^2  - 150 n - 69) x^{10}\\   
     + &x^{11}.
\end{align*}
 By Theorem \ref{thm:salem factor}, $\mu (f^*)$ is either a degree 22 Salem polynomial or divisible
 by a cyclotomic polynomial of degree at most $22$. There are only finitely many cyclotomic polynomials of a given degree.
 We can exclude the second case directly by computing the remainder after division for each such polynomial.
This proves Theorem \ref{thm:salem deg22}.
\section*{Acknowledgements} 
I thank my advisor Matthias Schütt who introduced me to this topic and the isotrivial fibration. 
Thanks to Davide Veniani and Víctor González-Alonso for discussions and helpful comments.

%
%
%


\begin{thebibliography}{}
 \bibitem{artin}Michael Artin Supersingular K3 surfaces, Ann. sci. É.N.S. (4), 7, 543-567 (1974).
 \bibitem{blanc cantat:non lift}Jérémy  Blanc and Serge Cantat : Dynamical degrees of birational transformations of projective surfaces; Jour. Amer. Math. Soc., to appear, preprint (2013). 
 \bibitem{esnault:algebraic entropy} Hélène Esnault, Vasudevan Srinivas, Algebraic versus topological entropy for surfaces over finite fields,Osaka J. Math. Volume 50, Number 3 (2013), 827-846.
 \bibitem{esnault oguiso:non liftability} Hélène Esnault, Keiji Oguiso, Non-liftability of automorphism groups of a K3 surface in positive characteristic, Math. Annalen (2015).
 \bibitem{esnault oguiso yu:maximum salemdeg} Hélène Esnault, Keiji Oguiso, Xun Yu, Automorphisms of elliptic K3 surfaces and Salem numbers of maximal degree,preprint (2014), \url{http://arxiv.org/abs/1411.0769}.
 \bibitem{Kodaira}Kunihiko Kodaira, On compact analytic surfaces II, Ann. Math., Vol. 77, No. 3 (1963), 563-626.
 \bibitem{kondo:elliptic pencil}Shigeyuki Kondo, Automorphisms of algebraic K3 surfaces which act trivially on the Picard lattice, J. Math. Soc. Japan Vol. 44, No. 1, 75-98 (1992).
 \bibitem{McMullen:siegel} Curtis Tracy McMullen, 2002. Dynamics on K3 surfaces: Salem numbers and Siegel disks. J. fur die Reine und Angew. Math. (2002)545: 201–233. 
 \bibitem{Ogus:supersingular}Arthur Ogus, Supersingular K3 crystals, Journées de Geometrie Algébrique de Rennes Vol. II, Ast. 64, 3-86 (1979).
 \bibitem{schuett:supersingular} Matthias Schuett, Dynamics on supersingular K3 surfaces, preprint (2015), \url{http://arxiv.org/abs/1502.06923}.
 \bibitem{shafarevic:elliptic pencil} Ilja Pjatetskij-Shapiro, Igor Rostislawowitsch, A torelli theorem of algebraic surfaces of Type K3, Izv. Akad. Nauk SSSR, Ser. Mat.,35:3 (1971), 530-572.
 \bibitem{shioda:mwl}Tetsuji Shioda, On the Mordell-Weil lattices, Com.. Math. Univ. St. Paul 39.2 (1990): 211-240.
 \bibitem{shimada:supersing} Ichiro Shimada, Transcendental lattices and supersingular reduction lattices of a singular K3 surface, Trans. of the AMS, Vol. 361, No. 2 (Feb., 2009), 909-949 .
 \bibitem{shimada:dynamics}Ichiro Shimada, Automorphisms of supersingular K3 surfaces and Salem polynomials, preprint (2015), \url{http://arxiv.org/abs/1503.04517}.
  \bibitem{silverman advanced topics}Joseph H. Silverman, Advanced Topics in the Arithmetic of Elliptic Curves, Springer (1999). 
 \bibitem{tate:algorithm} John Tate, Algorithm for determining the type of a singular fiber in an elliptic pencil, Modular Functions of One Variable IV, Lect. Notes in Math. Vol. 476, 1975, 33-52.
\end{thebibliography}
\end{document}